\documentclass[11pt]{article}
\usepackage[a4paper,hmargin=2.5cm,vmargin=2cm]{geometry}

\usepackage[round]{natbib}
\usepackage{hyperref}
\usepackage[misc]{ifsym}

\usepackage{amsmath}
\usepackage{mathtools}
\usepackage{amssymb}
\usepackage{enumitem}

\usepackage{fancyhdr}
\usepackage{lscape}
\usepackage{booktabs}
\usepackage{amsthm}
\theoremstyle{plain} 
\newtheorem{theorem}{Theorem}

\usepackage{array,longtable}

\usepackage{float}
\floatstyle{ruled}
\newfloat{procedure}{thp}{lop}[section]
\floatname{procedure}{Procedure}
\newfloat{algorithm}{thp}{lop}
\floatname{algorithm}{Algorithm}

\title{A Lagrangean Relaxation Algorithm for the Simple Plant Location Problem with Preferences}

\footnotesize\date{May 2018}

\author{
Xavier Cabezas\\
\footnotesize The University of Edinburgh\\
\footnotesize \texttt{J.X.Cabezas@sms.ed.ac.uk}\\
\footnotesize Escuela Superior Polit\'ecnica del Litoral\\
\footnotesize \texttt{joxacabe@espol.edu.ec}\\  
\and
Sergio Garc\'ia\\
\footnotesize The University of Edinburgh\\
\footnotesize \texttt{Sergio.Garcia-Quiles@ed.ac.uk} \\ 
}

\begin{document}

\maketitle

\begin{abstract} 
\noindent The Simple Plant Location Problem with Order (SPLPO) is a variant of the Simple Plant Location Problem (SPLP), where the customers have preferences over the facilities which will serve them. In particular, customers define their preferences by ranking each of the potential facilities. Even though the SPLP has been widely studied in the literature, the SPLPO has been studied much less and the size of the instances that can be solved is very limited. In this paper, we propose a heuristic that uses a Lagrangean relaxation output as a starting point of a semi-Lagrangean relaxation algorithm to find good feasible solutions (often the optimal solution). We carry out a computational study to illustrate the good performance of our method. 
\end{abstract}


\section{Introduction}
\label{intro}
The Simple Plant Location Problem with Order (SPLPO) is a variant of the Simple Plant Location Problem (SPLP), where the customers have preferences over the facilities which will serve them. In particular, customers define their preferences by ranking each of the potential facilities.

Let $I=\{1,\ldots\,m\}$ be a set of customers and $J=\{1,\ldots\,n\}$ a set of possible sites for opening facilities. Unit costs $c_{ij}\geq 0$ for supplying the demand of customer~$i$ from facility $j$ and costs $f_j \geq 0$ for opening a facility at $j$ are also considered. It is said that $k$ is \emph{i-worse} than $j$ if customer $i$ prefers facility $j$ to $k$ and it is written as $k < _ij$. We define $W_{ij}=\{k \in J \mid k < _ij\}$ as the set of facilities $k$ strictly \emph{i-worse} than $j$, its complement as $\overline{W_{ij}}$ and $W_{ij} \cup \{j\}$ as $W_{ij}'$. Let $x_{ij}$ be a decision  variable that represents the fraction of the demand required by customer $i$ and supplied by facility $j$. Since no capacities are considered, this demand can always be covered completely by one single facility. Therefore, we can guarantee that there is an optimal solution where values variables~$x_{ij}$ are in $\{0,1\}$. Let $y_j$ be a binary variable such that $y_j=1$ if a facility is open at location $j$ and $y_j=0$ otherwise.\\

The SPLPO formulation \citep{Canovas} is as follows:

\begin{align}
\hspace{1.2 cm} \text{Min} \qquad \qquad & \sum_{\mathclap{i \in I}}\sum_{\mathclap{j \in J}} c_{ij}x_{ij} + \sum_{j \in J} f_jy_j,& \nonumber  \\
\nonumber\\
\text{subject to} \hspace{0.5cm} & \sum_{\mathclap{j \in J}} x_{ij} = 1,  &  \forall i \in I, \hspace{0.8cm} \label{LimitFacilities} \\[0.05em]
& x_{ij} \leq y_j, & \hspace{0.05cm}  \forall i \in I, \forall j \in J, \hspace{0.8cm}  \label{OpenyIfx} \\[0.05em]
& \sum_{\mathclap{k \in \overline{W_{ij}}}} x_{ik} \geq y_j, & \forall i \in I, \hspace{0.05cm}  \forall j \in J, \hspace{0.8cm} \label{Preference} \\[0.05em]
& x_{ij} \geq 0, & \forall i \in I, \hspace{0.05cm}  \forall j \in J, \hspace{0.8cm} \label{Posi} \\[0.1em]
& y_j \in \{0,1\}, & \hspace{0.2cm} \forall j \in J. \hspace{0.8cm}\label{Bina}
\end{align}

Equalities (\ref{LimitFacilities}) ensure that every customer $i$ will be supplied by exactly one facility $j$, they are called \emph{assignment constraints}. Constraints (\ref{OpenyIfx}) ensure that if a costumer $i$ is supplied by a facility $j$ then $j$ must be opened, they are usually called \emph{variable upper bounds} (VUBs). Inequalities (\ref{Preference}) model the customers' preference orderings.

There are in the literature many studies on how to solve the SPLP using different methods and lagrangean relaxation is one of the most successful. For example, in \citet{Cornuejols1977}, the authors presented an application of Lagrangean relaxation to solve the SPLP in the context of location of bank accounts. They also proposed a heuristic algorithm and studied the Lagrangean dual to obtain lower and upper bounds. They also provided a bound for the relative error of these methods. \citet{Beltran2006} suggested, defined and applied the technique of semi-Lagrangean relaxation to the $p$-median problem. Some years later the study was extended to the SPLP, obtaining very good results. The method basically takes advantage of the linear formulation of the problems. First, it splits equality constraints $Ax=b$ into $Ax \leq b$ and $Ax \geq b$, and then relaxes the second one. The new model has the same objective value as the original problem (it closes the duality gap), but with the cost of making the new problem more difficult to solve. However, it has some properties that can be exploited. A summary of this method will be reviewed later. Another reduced form of this method was proposed by \citet{Monabbati}, which called it a surrogate semi-Lagrangean relaxation. A new algorithm for the dual problem using Lagrangean heuristics for both the original and surrogated version of the semi-Lagrangean relaxation can be found in \citet{Jornsten2016}.

The SPLPO has been studied much less and the main results are on finding new valid inequalities to strengthen the original formulation. \citet{Canovas} provided a new family of valid constraints which were used in combination with a preprocessing analysis. Another, more general, family of valid inequalities can be found in \citet{Vasilyev2013} with a polyhedral study. A branch and cut method was proposed by \citet{Vasilyev2010}.

Since these papers use exact methods that struggle to solve large instances, the aim of this paper is to develop a procedure that allows to solve the SPLPO efficiently in a heuristic way by using Lagrangean and semi-Lagrangean relaxation techniques. We propose a variable fixing heuristic that uses a Lagrangean relaxation output as the starting point of a semi-Lagrangean relaxation to find good feasible solutions (often the optimal solution).

The rest of the paper is organized as follows: Section \ref{rlr} provides a review of Lagrangean and semi-Lagrangean relaxation. SPLPO formulations for these two techniques are showed in Section \ref{lrSPLPO} and \ref{slrSPLPO}, along with the methods that will be used later to solve their respective dual problems. The complete method that we propose is presented in Section \ref{suAD}, where all the algorithms proposed in previous sections will be combined to build a heuristic procedure. In Section \ref{compr} we run a computational study that shows the good performance of our method. Finally some conclusions are given in Section \ref{conclusions}.

\section{Lagrangean and Semi-Lagrangean Relaxation Background} \label{rlr}

Let $P$ be a problem of the form:

$$min_x\{f(x)=cx \mid Ax \leq b, Cx \leq d, x \in X\},$$

\noindent where:

\begin{itemize}
\item The set $X$ may contain integrality constraints.
\item The family of constraints $Ax \leq b$ will be assumed complicated. (i.e., the problem $P$ without them is easier to solve).
\end{itemize}

The family $Ax \leq b$ can be placed in the objective function with vector of coefficients (Lagrangean multipliers) $\lambda$ which work as penalties when they are violated. It yields the Lagrangean relaxation problem $LR(\lambda)$ related to $P$ with multipliers $\lambda$:

\begin{equation}\label{LR}
LR(\lambda)=min_x\{f(x)+\lambda(Ax - b) \mid Cx \leq d, x \in X,\lambda \geq 0\},
\end{equation}

\noindent with its associated Lagrangean dual problem:

\begin{align}
LD_\lambda&=max_{\lambda \geq 0} min_x\{f(x)+\lambda(Ax - b) \mid Cx \leq d, x \in X,\lambda \geq 0\} \nonumber \\
&=max_{\lambda \geq 0} LR(\lambda).
\end{align}

Let $v(p)$ be the optimal objective function value for a particular problem $p$. For all $x \in X$ feasible for $P$ and any $\lambda \geq 0$, $f(x)+\lambda(Ax - b) \leq f(x)$. Therefore, $v(LD_\lambda) \leq v(P)$ and $v(LR(\lambda)) \leq v(P)$. Furthermore, if we supposed that $\{x \in X \mid Cx \leq X\}$ is a bounded polyhedron, there is a finite family of extreme points $x$ in its the convex hull and therefore the objective function of $LR(\lambda)$ is a piecewise linear concave function on $\lambda$. Thus, there are some points where the function is not differentiable. Additionally, it can be proved that if $x^*$ be an optimal solution for $LR(\lambda^*)$, $Ax^*-b$ is a subgradient of $LR(\lambda)$ at point $\lambda^*$.

Also, if $LP(P)$ is problem $P$ without integrality constraints (that is, its linear relaxation), then $v(LP(P)) \leq v(LD_\lambda) \leq v(P)$. And, $v(LP(P)) = v(LD_\lambda)$ whenever $v(LP(LD_\lambda))=v(LD_\lambda)$ (which is called integrality property).\\

Consider now the problem $P'$:
 
$$min_x\{f(x)=cx \mid Ax=b, x \in X\},$$

\noindent where:

\begin{itemize}
\item The set $X$ contains integrality constraints.
\item $A$, $b$ and $c$ are non-negative.
\end{itemize}

The semi-lagrangean relaxation problem $SLR(\lambda)$ related to $P'$ is:
\begin{equation}\label{SLR}
SLR(\lambda)=min_x\{f(x)+\lambda(b-Ax) \mid Ax \leq b, x \in X,\lambda \geq 0\}.
\end{equation}

After having split $Ax=b$, inequality $Ax \geq b$ has been relaxed with a vector multiplier $\lambda$ whereas inequality $Ax \leq b$ has been kept. Its semi-Lagrangean dual problem is then:

\begin{equation}
SLD_\lambda=max_{\lambda \geq 0} SLR(\lambda).
\end{equation}

It is clear that $v(LR(\lambda)) \leq v(SLR(\lambda))$ because $SLR(\lambda)$ has more constraints. Thus, $v(LP(P')) \leq v(LD_\lambda) \leq v(SLD_\lambda) \leq v(P')$. \citet{Beltran2006} proved that the semi-Lagrangean dual problem has the same optimal value than $P'$, i.e., $v(P')=v(SLD_\lambda)$. They also proved that the objective function of $SLR(\lambda)$ is concave, non-decreasing on its domain and $b-Ax$ is a subgradient at point $\lambda$. Moreover, there is an interval $[\lambda^*,+ \infty)$ where with any of its elements we met the same optimal solution of $SLR(\lambda)$. For details, see \citep{Geoffrion1974,Fisher2004,Beltran2006,Beltran2009}.

\section{A Lagrangean Relaxation for SPLPO}\label{lrSPLPO}

In this section we consider a Lagrangean Relaxation for the SPLPO. Our first result shows that the proposed model is easy to solve.

If we relax constraints (\ref{LimitFacilities}) and (\ref{Preference}), then they are moved to the objective function with penalty coefficients (multipliers) when they are violated, thus obtaining the following Lagrangean relaxation problem $LR(\mu,\lambda)$:

\begin{eqnarray*}
\min_{(x,y)} \hspace{0.2 cm}  \sum_{\mathclap{i}}\sum_{\mathclap{j}} c_{ij}x_{ij} + \sum_{j} f_jy_j + \sum_i \mu_i \left( 1-\sum_j x_{ij} \right)   +  \sum_{\mathclap{i}} \sum_{\mathclap{j}} \lambda_{ij} \left[y_j - \sum_{\mathclap{k \in \overline{W_{ij}}}} x_{ik} \right], \\
= \min_{(x,y)} \hspace{0.2 cm} \sum_i\sum_j \left(c_{ij}-\mu_{i} \right) x_{ij} - \sum_i\sum_j \lambda_{ij} \sum_{\mathclap{k \in \overline{W_{ij}}}} x_{ik} + \sum_j \left( f_j + \sum_i \lambda_{ij} \right) y_j + \sum_i \mu_i ,
\end{eqnarray*}

\noindent $\text{subject to:} \hspace{0.2 cm} (\ref{OpenyIfx}), (\ref{Posi}) \hspace{0.2 cm} \text{and} \hspace{0.2 cm} (\ref{Bina}).$\\

\noindent The multiplier vectors $\mu$ and $\lambda$ in $LR(\mu,\lambda)$ are unrestricted in sign and nonnegative, respectively.

Let $F(P)$ be the set of feasible solutions of problem $P$. Then for all $(x,y) \in F(\text{SPLPO})$ the objective function of $LR(\mu,\lambda)$ evaluated in $(x,y)$ is always less than or equal to the objective function of $P$ evaluated in $(x,y)$. Therefore  $v(LR(\mu,\lambda)) \leq v(\text{SPLPO}) $.

In order to obtain the best lower bound for SPLPO, we need to solve the following Lagrangean dual problem: 

$$LD_{\mu\lambda} = \max_{\mu \in \mathbb{R},\lambda \geq 0} LR(\mu,\lambda).$$

Suppose that each customer $i$ ranks the different potential facilities $j$ with a number $p_{ij} \in \{1,\ldots,n\}$ with $1$ and $n$ the most and the least preferred, respectively. Since each multiplier $\lambda_{ij}$ in a term of $\sum_i\sum_j \lambda_{ij} \sum_{k \in \overline{W_{ij}}} x_{ik}$ in $LR(\mu,\lambda)$ will be multiplied by a sum of $p_{ij}$ variables $x_{ik}$ with $k\geq_ij$, then each $x_{ij}$ will be multiplied by a sum of $(n-p_{ij}+1)$ values $\lambda_{ik}$ with $k \leq _ij$. Therefore:

$$\sum_i\sum_j \lambda_{ij} \sum_{\mathclap{k \in \overline{W_{ij}}}} x_{ik} = \sum_i \sum_j \Bigg(\hspace{0.5 cm}   \sum_{\mathclap{\substack{k \in W'_{ij}\\|W'_{ij}|=n-p_{ij}+1}}} \lambda_{ik} \Bigg) x_{ij},$$

\noindent and $LR(\mu,\lambda)$ can be rewritten as:

\begin{equation*}
LR(\mu,\lambda) = \min_{(x,y)} \hspace{0.2 cm} \sum_i\sum_j \Bigg( c_{ij}-\mu_{i} -  \sum_{\mathclap{\substack{k \in W'_{ij}}}} \lambda_{ik} \Bigg) x_{ij} + \sum_j \left( f_j + \sum_i \lambda_{ij} \right) y_j + \sum_i \mu_i,
\end{equation*}

\noindent $\text{subject to:} \hspace{0.2 cm} (\ref{OpenyIfx}), (\ref{Posi}) \hspace{0.2 cm} \text{and} \hspace{0.2 cm} (\ref{Bina})$.\\

As in \citet{Cornuejols1977}, this Lagrangean problem is easy to solve analytically for fixed vectors $\mu$ and $\lambda$. If we define $\Lambda_{ij}$ as:

$$\Lambda_{ij}=\sum_{\mathclap{\substack{k \in W'_{ij}}}} \lambda_{ik},$$

\noindent then we have the following result:

\begin{theorem}\label{thmsgm}

An optimal solution for $LR(\mu,\lambda)$ can be obtained as follows: 

\begin{equation*}
y_j = \left\{
\begin{array}{rl}
1, & \textup{if } \sum_i \min (0,c_{ij}-\mu_i-\Lambda_{ij})+\left( f_j + \sum_i \lambda_{ij} \right) < 0,\\
0, & \textup{otherwise}.
\end{array} \right.
\end{equation*}

and,

\begin{equation*}
x_{ij} = \left\{
\begin{array}{rl}
1, & \textup{if } y_j=1 \textup{ and } (c_{ij}-\mu_i-\Lambda_{ij}) < 0,\\
0, & \textup{otherwise}.
\end{array} \right.
\end{equation*}

\end{theorem}

\begin{proof}

We have that:

\begin{eqnarray*}
LR(\mu,\lambda) &=& \min_{(x,y)} \hspace{0.2 cm} \sum_i\sum_j \Bigg( c_{ij}-\mu_{i} - \Lambda_{ij} \Bigg) x_{ij} + \sum_j \left( f_j + \sum_i \lambda_{ij} \right) y_j + \sum_i \mu_i, \\
& = &  \min_{(x,y)} \hspace{0.2 cm} \sum_j \left[ \sum_i \Bigg( c_{ij}-\mu_{i} - \Lambda_{ij} \Bigg) x_{ij} + \left( f_j + \sum_i \lambda_{ij} \right) y_j \right] + \sum_i \mu_i,
\end{eqnarray*}

\noindent $\text{subject to:} \hspace{0.2 cm} (\ref{OpenyIfx}), (\ref{Posi}) \hspace{0.2 cm} \text{and} \hspace{0.2 cm} (\ref{Bina}).$\\

As a consequence of constraints (\ref{OpenyIfx}) and for fixed vectors $\mu$ and $\lambda$, the optimal values for $x_{ij}$ will be $x_{ij}=1$ if $y_j=1$ and $c_{ij}-\mu_i-\Lambda_{ij} < 0$, otherwise $x_{ij}=0$. Then, if we define $\rho_j(\mu,\lambda)=\sum_i \min (0,c_{ij}-\mu_i-\Lambda_{ij})+\left( f_j + \sum_i \lambda_{ij} \right)$, then the optimal vector $y$ can be obtained by solving the following minimization problem:

\begin{equation*}
\min \sum_j \rho_j(\mu,\lambda) y_j,
\end{equation*}

\noindent $\text{subject to:} \hspace{0.2 cm} y_j \in \{0,1\}.$\\

The solution to this problem is straightforward.
\end{proof}

An issue with this model is that $LR(\mu,\lambda)$ has the integrality property, that is, its optimal value is equal to the standard linear relaxation $LP(\text{SPLPO})$. Furthermore, the values obtained of the function $LR(\mu,\lambda)$ during the search of the solution for $LD_{\mu\lambda}$ are infeasible to the original problem SPLPO. So, as an alternative, we propose to use the solution of this problem as a starting point for another procedure that allows us to find feasible solutions to SPLPO.

\subsection{Subgradient Method for the Lagrangean Dual $LD_{\mu\lambda}$} \label{smSPLPO}

The subgradient method was originally proposed by \citet{Held1971} and validated by \citet{Held1974}. Given multipliers $\lambda$ and $\mu$, this method tries to optimize $LD$ by taking steps along a subgradient of $LR(\mu,\lambda)$. A sketch of the whole procedure is given in Algorithm \ref{SGM}. 

\begin{algorithm}[ht!]
\caption{Subgradient method (SG) for SPLPO.}
\label{SGM}
\vspace{0.20 cm}
\footnotesize

Let $LD=max_{(\mu,\lambda)}LR(\mu,\lambda)$ with $\mu \in \mathbb{R}$ and $\lambda \geq 0$.\\ 

\vspace{0.3 cm} 

Step 1. \textbf{(Initialization)}. Let $LR_{AIM}$ by a heuristic method. Set $\beta=2$. Let $k$ be an integer number. Let $q$ be a number in $[0,1]$. Let $[\mu^0,\lambda^0]$ be a starting point. \\

Step 2. \textbf{(Obtaining values $x^0_{ij}, \forall i,j$ and $y^0_j, \forall j$)}. Find $LR_{best}^0=LR(\mu^0,\lambda^0)$. Set $iter=0$.\\

Step 3. \textbf{(Finding a subgradient)}. Find a subgradient $s^{iter}$ for $LR(\mu^{iter},\lambda^{iter})$.\\

Step 4. \textbf{(Stop criterion)}. If $s^{iter}=0$, STOP and $[\mu^{iter},\lambda^{iter}]$ is optimal.\\
Otherwise, go to Step 5.\\

Step 5. \textbf{(Step size)}. Let $\alpha^{iter}=\beta\frac{LR_{AIM}-LR(\mu^{iter},\lambda^{iter})}{\|s^{iter}\|_2^2}$.\\
If $LR(\mu^{iter},\lambda^{iter})$ does not improve for $k$ consecutive iterations, set $\beta=\beta*q$.\\

Step 6. \textbf{(Updating multipliers)} Set\\
$[\mu^{iter+1},\lambda^{iter+1}]=[(\mu^{iter}+\alpha^{iter}s_{\mu}^{iter}),max(0,\lambda^{iter}+\alpha^{iter}s_{\lambda}^{iter})]$.\\

Step 7. \textbf{(Updating incumbent)}. Let\\
$LR_{best}^{iter+1}=max\{LR_{best}^{iter},LR(\mu^{iter+1},\lambda^{iter+1})\}$. Update $iter=iter+1$.\\

Step 8. \textbf{(Stop criterion)}. If $iter=MAXiter$, STOP. Otherwise go to Step 3.

\end{algorithm}

As seen in Steps 2 and 5, at each iteration, we need to solve an instance of $LR(\mu,\lambda)$. Each of them is easy to solve, as shown in Theorem~\ref{thmsgm}. However, it is important to note that the whole procedure could still be slow, that is, slower than solving $LP(\text{SPLPO})$ with a commercial linear programming solver.

It is easy to see that:

$$s=\begin{bmatrix} {\displaystyle  1-\sum_{j \in J} x_{ij}; \forall i} \\ {\displaystyle y_j-\sum_{k \in \overline{W_{ij}}} x_{ik}; \forall i,j} \end{bmatrix} = \begin{bmatrix} s_{\mu} \\ s_{\lambda}\end{bmatrix}  \in \mathbb{R}^{m(n+1)},$$

\noindent is a subgradient for $LR(\mu,\lambda)$. If this vector is $0$, then the procedure ends.

In our computational experiments, that will be showed later, an upper bound for $LR_{AIM}$ in Step 5 is found using a simple heuristic. See Algorithm \ref{h2}. First, it opens a facility that supplies all customers with the lowest operating cost and it is removed from the set $J'=J$. Then, for each unopened one, each customer compares and chooses the most preferred facility between it and its previously assigned supplier. The new cost is saved. The new open facility is the one with the lowest operating cost. It is removed from $J'$. This is repeated until $J'=\{\}$.

\begin{algorithm}[ht!]
\caption{Heuristic to find an upper bound for SPLPO (Hc).}
\label{h2}
\vspace{0.20 cm}
\footnotesize

Let $G(I \times J,E=\{\})$ be a bipartite graph.\\

\vspace{0.3 cm}

Step 1. Find $j_0 \in J$ such that $\sum_i c_{ij_0}=min\{\sum_i c_{i1},\ldots,\sum_i c_{in}\}$.

\hspace{0.95 cm} Set $J'=J\setminus \{j_0\}$, $E=\{(i,j_0)\}$ and $TC_{prev}=\sum_i c_{ij_0}$.\\

Step 2. For all $j \in J'$, do:

\hspace{0.95 cm} Set $E_j=\{\}$.

\hspace{1.5 cm}  For all $i \in I$, do:

\hspace{1.5 cm}  Find $k_{pref}$ such that $p_{ik_{pref}}=min[\{p_{ij}\} \cup \{p_{ik} \mid (i,k) \in E\}]$.

\hspace{1.5 cm} Set $E_j= E_j \cup \{(i,k_{pref})\}$.

\hspace{1.5 cm} Compute $TC_j=\sum_{(i,j) \in E_j} c_{ij}$. \\

Step 3. Find $j_0$ such that $TC_{j_0}=min\{ TC_j \mid j=1,\ldots,n \}$. Set $J'=J' \setminus \{j_0\}$ and $E=E_{j_0}$.\\

Step 4. If $J'=\{\}$, STOP. Otherwise go to Step 2.

\end{algorithm}

We also tried another heuristic. It is basically the same that for Hc with a different Step 4 which allows the algorithm to stop earlier. We name it Hs (see Algorithm \ref{Hs}) and its results will be reported later.

\begin{algorithm}[ht!]
\caption{Hs.}
\label{Hs}
\vspace{0.20 cm}
\footnotesize
Step 1..3. Same as Hc.\\

Step 4. If $CT_{j_0} \geq CT_{prev}$ then STOP. Otherwise, set $CT_{prev}=CT_{j_0}$ and go to Step 2.
\end{algorithm}

It is possible, due to an overestimation of $LR_{AIM}$, that the function $LR(\mu,\lambda)$ does not improve for many iterations. This can be overcome by setting the para\-meter $\beta$ to a fixed value (for example, 2) and reducing it slowly.

In Step 6, we need vector $\lambda$ to be nonnegative, therefore in the nonne\-gative orthant, i.e.,  $[\lambda_i]^+=max\{0,\lambda_i\}$, for all of its components. $\mu$ must remain unchanged as it is an unrestricted vector.

Further details can be found in \citep{GuignardTutorial,IntegerProgramming} and \citet{Poljak} for the step size in Step 5 in Algorithm \ref{SGM}.

\section{A Semi-Lagrangean Relaxation for SPLPO} \label{slrSPLPO}

Now we use the technique of the Semi-Lagrangean relaxation, as this leads to close the duality gap. The equality constraints (\ref{LimitFacilities}) have been split into two inequalities $\sum_{j \in J} x_{ij} \leq 1$ and $\sum_{j \in J} x_{ij} \geq 1$ to obtain the following model:

\begin{eqnarray*}
SLR(\gamma) & = & \min_{(x,y)} \hspace{0.2 cm}  \sum_{\mathclap{i}}\sum_{\mathclap{j}} c_{ij}x_{ij} + \sum_{j} f_jy_j + \sum_i \gamma_i \left( 1-\sum_j x_{ij} \right),\\
& = & \min_{(x,y)} \hspace{0.2 cm} \sum_i\sum_j \left(c_{ij}-\gamma_{i} \right) x_{ij} + \sum_j f_j y_j + \sum_i \gamma_i,\\
& = & \min_{(x,y)} \hspace{0.2 cm} \sum_j \left(\sum_i \left(c_{ij}-\gamma_{i} \right) x_{ij} +  f_j y_j \right) + \sum_i \gamma_i,
\end{eqnarray*}

\noindent $\text{subject to:} \hspace{0.2 cm} (\ref{OpenyIfx}), (\ref{Preference}), (\ref{Posi}), (\ref{Bina})$, and $\sum_{j \in J} x_{ij} \leq 1$. Every component of the multipliers vector $\gamma$ is nonnegative.\\

As mentioned in Section \ref{rlr}, the objective function is concave and non-decreasing in its domain. Therefore, its semi-Lagrangean dual problem:

$$SLD_\gamma = \max_{\gamma \geq 0} SLR(\gamma),$$ 

\noindent can be solved using an ascent method. Also, as pointed out, there is a set $Q=[\gamma^*,+ \infty)$ such that, for any $q \in Q$ the optimum of $SLR(\gamma)$ is met.

For the SPLP (no preferences), \citet{Beltran2009} proved that there is a closed interval where the search of the multipliers could be done. Following the same idea, we provide the next two results for the SPLPO:

\begin{theorem} \label{ubbox}
Let $cp_i=max_j\{c_{ij}+f_j\}$ and let $cp=(cp_1,\ldots,cp_m)$ be the maximum of the costs for each costumer $i$ associated to each facility $j$ and the vector of these costs, respectively. If $\gamma \geq cp$, then $\gamma \in Q$.
\end{theorem}

\begin{proof}
As we know, the semi-Lagrangean relaxation closes the duality gap if for all $i$, $\sum_{j \in J} x_{ij}=1$. Assume that $SLR(\gamma)<\infty$, otherwise the proposition is trivially true. By hypothesis, $cp_i-\gamma_i \leq 0$ for all $i \in I$. If we choose $j'$ such that $cp_i=c_{ij'}+f_{j'}$, it turns out in $(c_{ij'}-\gamma_i)+f_{j'} \leq 0$, and this inequality is true for any $j$ since $j'$ gives the maximum among all $c_{ij}+f_j$. Therefore, the event $\sum_{j \in J} x_{ij}=0$ can not happen at an optimal solution, because it is always possible set $x_{ij}=1$ and $y_{j}=1$ for all $i$ and $j$ meeting all the constraints. 
\end{proof}

\begin{theorem} \label{lbbox}
For each $i \in I$, let $c_i^1 \leq \ldots \leq c_i^{n}$ be the sorted costs $c_{ij}$. If $\gamma < c^1$ then $\gamma \notin Q$.
\end{theorem}

\begin{proof}
By hypothesis $c^1_i - \gamma_i >0 $, then $c^{j}_i - \gamma_i >0 $  for all $j \in J$. In that case, $x^*_{ij}=0$ for all $j \in J$ in any optimal solution $x^*_\gamma$. Therefore, $\sum_{j \in J} x_{ij}=1$ can not happen at an optimal solution and $\gamma \notin Q$. 
\end{proof}

The result of Theorem \ref{ubbox} is weaker than the obtained by \citet{Beltran2009} for SPLP in the following sense. They choose $cp_i$ equal to $min_j\{c_{ij}+f_j\}$, but with this, there is no guarantee that the preference constraints hold. We can set $x^*_{ij'}=1$ and  $y_{j'}=1$ for the particular $j'$ such that $min_j\{c_{ij}+f_j\}=c_{ij'}+f_{j'}$ but not necessarily for all $j$, as in this way, are not available all the possible combinations of $x$'s and $y$'s that consider all customers preferences. Therefore, we are taking the risk of obtaining a non optimal solution.

The interval of search for $\gamma$ components is then:

\begin{equation}
B=\{\gamma \mid c^1 < \gamma \leq cp\}. 
\end{equation}

Using the sorted costs $c_i^1 \leq \ldots \leq c_i^{n}$, each component $\gamma_i$ of $\gamma$ can be either in an interval of the form $I_j=(c_i^j,c_i^{j+1}]$ where $j \in \{1,\ldots,m-1\}$ or out of it. For the first case, there are infinite values of $\gamma_i$ that can belong to a single interval $(c_i^j,c_i^{j+1}]$, but each of them has the same effect on the optimal value of $SLR(\gamma)$. This is because going from an interval $I_j$ to the next $I_{j+1}$ could imply a change in the choice of the arc $(i,j)$ to the arc $(i,j+1)$ in the bipartite graph $G_{\gamma}=(I \times J,E=\{(i,j)\mid x_{ij}=1 \text{ in the solution of } SLR(\gamma) \})$, which means a change in the solution. Hence, we just need a single $\gamma_i$ representative of the intervals. As $\gamma$ goes to infinity all combined costs $(c_{ij}-\gamma_i)+f_{j}$ become negative, and hence the semi-Lagrangean relaxation problem can be as difficult to solve as the original SPLPO problem. Then, it is always convenient to choose a $\gamma_i \in I_i$ as smaller as possible, that is, at an epsilon $\epsilon$ distance from the lower bound of an interval. Also, it is easy to check that $c_i^{n}$ is always less than $cp_i$ if $f_j \geq 0$. These ideas will be applied in the ascent method used later.

\subsection{Dual Ascent Method for the Semi-Lagrangean Dual $SLD_{\gamma}$} \label{daSPLPO}

In general, a dual ascent algorithm modifies a current value of multipliers in order to achieve a steady increase in $SLR(\gamma)$, see \citet{Bilde1977}. In our case this is possible due to the aforementioned properties. The method we used is explained under Algorithm \ref{DAM}.   

\begin{algorithm}[ht!]
\caption{Dual ascent method (DA) for $SLD_\gamma$.}
\label{DAM}
\vspace{0.20 cm}
\footnotesize

Let $c_i^1 \leq \ldots \leq c_i^{n}$ be the sorted costs $c_{ij}$ and let $\gamma^0$ be an initial vector $\gamma$.\\

\vspace{0.10 cm} 

Step 1. \textbf{(Initialization)}. Set $\epsilon>0$, $iter=0$. If $k>n$, $c_i^{k}=cp_i$. 

\hspace{1.0 cm} For all $i \in I$ do:

\hspace{1.0 cm} $cp_i=max_j\{c_{ij}+f_j\}$,

\hspace{1.5 cm} If $\gamma_i^0<c_i^1$, then $\gamma_i^0=c_i^1+\epsilon$,

\hspace{1.5 cm} Else,

\hspace{1.5 cm} If $\gamma_i^0>c_i^{n}$, then $\gamma_i^0=c_i^{n}$,

\hspace{1.5 cm} Else,

\hspace{1.5 cm} find a $j_i$ such that $\gamma_i$ belongs to $(c_i^{j_i},c_i^{j_i+1}]$, and set $\gamma^0_i=c_i^{j_i}+\epsilon$.\\

Step 2. \textbf{(Obtaining values $x^{iter}_{ij}$ and $y^{iter}_j$)}. Solve $SLR(\gamma^{iter})$.\\

Step 3. \textbf{(Finding a subgradient)}. Find a subgradient $s^{iter}$ of $SLR(\gamma^{iter})$.\\

Step 4. \textbf{(Stop criterion)}. If $s^{iter}=0$, STOP and $\gamma^{iter}$ is optimal. Otherwise, go to Step 5.\\

Step 5. \textbf{(Updating multipliers)} For each $i$ such that $s^{iter}_i=1$, $j_i=j_i+1$ and\\ $\gamma^{iter+1}_i=min\{c_i^{j_i} + \epsilon,cp_i\}$. Set $iter=iter+1$. Go to step 2.

\end{algorithm}

In Step 1, the algorithm gives an appropriate position of the terms of an initial multipliers vector in the intervals $I_j$. In Step 2, the procedure needs to solve a sequence of problems $SLR(\gamma)$, but due to the preference constraints, they are not as easy to solve as in the Lagrangean relaxation case $LR(\mu,\lambda)$ proposed before. However, it is possible to set some variables $x_{ij}$ equal to $0$ in advance, in order to make it easier. Every $x_{ij}$ must be $0$ if $c_{ij} - \gamma_i>0$.

A subgradient for $SLR(\gamma)$ is computed in Step 3 as:

$$s= \left[ 1-\sum_{j \in J} x_{ij}; \forall i \right]=\left[ s_\gamma \right] \in \mathbb{R}^m,$$

\noindent where each component of $s$ belongs to $\{0,1\}$ as the $\sum_{j \in J} x_{ij} \leq 1$ must be satisfied.

In Step 4, a stop criterion is given. Finally, the multipliers are updated (increased) by jumping from the current to the next $I_j$ interval for each component~$i$ in $\gamma$.

\section{Speeding Up the Search for the Optimal Solution} \label{suAD}

After a certain number of iterations of DA that has been fixed beforehand, we apply a variable fixing heuristic VFH that takes, among all $y_j=1$, a percentage of them by sorting this set of $y_j$'s by a determined criterion. Then, it sets all the selected $y_j$ equal to $1$. Finally the subproblem is solved. The method is described under Algorithm \ref{FVH}.

\begin{algorithm}[h!]
\caption{Variable fixing heuristic (VFH) to speed up DA.}
\label{FVH}
\vspace{0.20 cm}
\footnotesize

Let $SLR_{\gamma}$ be an instance of DA after $k$ iterations. \\
Let $Y_{\gamma}$ be a set of $y_j$'s equal to 1 at instance $SLR_{\gamma}$.

\vspace{0.3 cm} 

Step 1. Set $ps \in [0,1] \in \mathbb{R}$. Sort the elements of $Y_{\gamma}$ such that $y_j \prec y_k$ if $\sum_i(c_{ij}+f_j) < \sum_i(c_{ik}+f_k)$.\\

Step 2. Choose the first $ps \times 100\%$ smallest elements from the sorted $Y_{\gamma}$ to form the set $Y_{\gamma}^{subset}$.\\

Step 3. Set, $y_j=1$ for all $y_j \in Y_{\gamma}^{subset}$.\\

Step 4. Solve the original $SPLPO$.

\end{algorithm}

Step 1 shows the criterion that we used to sort variables $y_j$. We also tried to sort it by $y_j \prec y_k$ if $\sum_ip_{ij} < \sum_ip_{ik}$, where the $p$'s are the preferences. We obtained similar results, but the option proposed in Algorithm \ref{FVH} was slightly better. 

Another important issue is the starting point for the dual ascent method. We have tried two different starting $\gamma$ vectors, $\gamma=0$ and $\gamma$ equal to $\mu$, which is one of the multipliers that can be found using the subgradient method SG for $LD_{\mu\lambda}$ after a certain number of iterations. This $\mu$ is the penalization associated with the relaxed family of constraints (\ref{LimitFacilities}). We obtained better results with the second approach. This is because each sub-problem $LR(\mu,\lambda)$ to be solved in the optimization of $LD_{\mu\lambda}$ is easy, see Theorem \ref{thmsgm}. For SG we used $\mu_i^0=min_j\{c_{ij}+f_j\}$ for all $i \in I$ and $\lambda_{ij}=0$ for all $i \in I$ and $j \in J$.
 
The whole accelerated dual ascent procedure is presented in Algorithm \ref{complete}.

\begin{algorithm}[h!]
\caption{Accelerated dual ascent algorithm. (ADA).}
\label{complete}
\vspace{0.20 cm}
\footnotesize


Step 1. Set $\mu_i^0=min_j\{c_{ij}+f_j\}$ for all $i \in I$ and $\lambda_{ij}=0$ for all $i \in$ and $j \in J$.\\

Step 2. Run $\text{SG}(\mu^0,\lambda^0)$ during $sg\_iter$ iterations. Find the iteration $best\_sg\_iter$  where is the best value of $LR(\mu,\lambda)$. Set $\gamma^0=\mu^{best\_sg\_iter}$.\\

Step 3. During $da\_iter$ iterations, run $\text{DA}(\gamma^0)$.\\

Step 4. During $vfh\_iter$ run $\text{DA}$. Get $Y_{\gamma}$ and run VFH. Save the solution. \\

Step 5. Find the best solution in the history and get best values for $x_{ij}$ and $y_j$.

\end{algorithm}

\section{Computational results} \label{compr}

In this section we present the computational results obtained after having applied the algorithms shown before. The experiments have been carried out on a PC with Intel$\textsuperscript{\textregistered}$ Xeon$\textsuperscript{\textregistered}$ 3.40 GHz processor and 16 Gb of RAM under a Windows$\textsuperscript{\textregistered}$ 7 operative system. All the procedures and algorithms have been written using the version 4.0.3 of the Mosel Xpress language and when a MIP problem needed to be solved we used FICO Xpress$\textsuperscript{\textregistered}$ version 8.0.

The instances were, at first, taken from \cite{Canovas}, which are based on the Beasley's OR-Library \citep[see][]{Beasley}. These are:\\

\noindent 131, 132, 133, 134: $m=50$ and $n=50$\\
a75\_50, b75\_50, c75\_50: $m=75$ and $n=50$\\
a100\_75, b100\_75, c100\_75: $m=100$ and $n=75$\\

Additional data sets were generated using the same algorithm proposed in \cite{Canovas}:\\

\noindent a125\_100, b125\_100, c125\_100: $m=125$ and $n=100$,\\
a150\_100, b150\_100, c150\_100: $m=150$ and $n=100$.\\

The headers of the tables have the following meanings:

\begin{itemize}[noitemsep]
\item Prob: Name of the problem.
\item Opt: Optimal objective function value of the problem P.
\item y\_j: Number of opened facilities.
\item bestUB: Best upper bound.
\item GAP$_{\text{o}}$ (\%): $\text{bestUB-Opt}$ ($\frac{\text{bestUB-Opt}}{\text{Opt}} \times 100\%$). The absolute and relative gap between the optimal and the value of an algorithm.
\item GAP$_{\text{LP}}$: $\frac{\text{LP(P)-SG}}{\text{LP(P)}} \times 100\%$. The relative gap between the linear relaxation and the lower bound of the Lagrangian relaxation.
\item LP(P): Linear relaxation value for problem P.
\item SG: Best value using the subgradient method for $LD_{\mu\lambda}$.
\item iter (itbsol.): Number of iterations (iteration where best solution was found).
\item t (Tt): CPU time in seconds (Total time).
\item sol10\%: First solution found that is within 10\% or less of the optimal solution.
\item UB (LB): Upper bound (Lower bound).
\end{itemize}

\begin{table}[h!]
  \centering
  \caption{Comparison between upper bounds heuristics Hs and Hc.}
\setlength{\tabcolsep}{2.8pt} 
\renewcommand{\arraystretch}{1.3} 
  \footnotesize
    \begin{tabular}{rrrrrrrrrrrrr}
    \toprule
        \multicolumn{1}{c}{} & \multicolumn{1}{c}{} & \multicolumn{1}{c}{} & \multicolumn{4}{c}{\textbf{Hs}} & \multicolumn{4}{c}{\textbf{Hc} }\\
        \cmidrule(l){4-7}
        \cmidrule(l){8-11}
    \multicolumn{1}{c}{\textbf{Prob}} & \multicolumn{1}{c}{\textbf{Opt}} & \multicolumn{1}{c}{\textbf{y\_j}} & \multicolumn{1}{c}{\textbf{bestUB}} & \multicolumn{1}{c}{\textbf{GAP$_{\text{o}}$}} & \multicolumn{1}{c}{\textbf{GAP$_{\text{o}}$\%}} & \multicolumn{1}{c}{\textbf{y\_j}} & \multicolumn{1}{c}{\textbf{bestUB}} & \multicolumn{1}{c}{\textbf{GAP$_{\text{o}}$}} & \multicolumn{1}{c}{\textbf{GAP$_{\text{o}}$\%}} & \multicolumn{1}{c}{\textbf{y\_j}}\\
    \midrule
    131\_1 & 1001440 & 6     & 1001440 & 0     & 0.00\% & 6     & 1001440 & 0     & 0.00\% & 6 \\
    131\_2 & 982517 & 9     & 1003100 & 20583 & 2.09\% & 5     & 982517 & 0     & 0.00\% & 9 \\
    131\_3 & 1039853 & 10    & 1248143 & 208290 & 20.03\% & 1     & 1139012 & 99159 & 9.54\% & 9 \\
    131\_4 & 1028447 & 9     & 1248143 & 219696 & 21.36\% & 1     & 1049791 & 21344 & 2.08\% & 7 \\
    132\_1 & 1122750 & 8     & 1248143 & 125393 & 11.17\% & 1     & 1239961 & 117211 & 10.44\% & 7 \\
    132\_2 & 1157722 & 9     & 1248143 & 90421 & 7.81\% & 1     & 1199057 & 41335 & 3.57\% & 9 \\
    132\_3 & 1146301 & 6     & 1248143 & 101842 & 8.88\% & 1     & 1146301 & 0     & 0.00\% & 6 \\
    132\_4 & 1036779 & 5     & 1248143 & 211364 & 20.39\% & 1     & 1036779 & 0     & 0.00\% & 5 \\
    133\_1 & 1103272 & 7     & 1248143 & 144871 & 13.13\% & 1     & 1103272 & 0     & 0.00\% & 7 \\
    133\_2 & 1035443 & 5     & 1248143 & 212700 & 20.54\% & 1     & 1100713 & 65270 & 6.30\% & 7 \\
    133\_3 & 1171331 & 6     & 1248143 & 76812 & 6.56\% & 1     & 1208198 & 36867 & 3.15\% & 4 \\
    133\_4 & 1083636 & 9     & 1248143 & 164507 & 15.18\% & 1     & 1090582 & 6946  & 0.64\% & 9 \\
    134\_1 & 1179639 & 4     & 1248143 & 68504 & 5.81\% & 1     & 1179639 & 0     & 0.00\% & 4 \\
    134\_2 & 1121633 & 7     & 1248143 & 126510 & 11.28\% & 1     & 1205809 & 84176 & 7.50\% & 5 \\
    134\_3 & 1171409 & 6     & 1248143 & 76734 & 6.55\% & 1     & 1173693 & 2284  & 0.19\% & 7 \\
    134\_4 & 1210465 & 3     & 1248143 & 37678 & 3.11\% & 1     & 1248143 & 37678 & 3.11\% & 1 \\
    \midrule
    a75\_50\_1 & 1661269 & 7     & 1787955 & 126686 & 7.63\% & 1     & 1787955 & 126686 & 7.63\% & 1 \\
    a75\_50\_2 & 1632907 & 6     & 1784848 & 151941 & 9.30\% & 2     & 1779576 & 146669 & 8.98\% & 4 \\
    a75\_50\_3 & 1632213 & 7     & 1738404 & 106191 & 6.51\% & 3     & 1738404 & 106191 & 6.51\% & 3 \\
    a75\_50\_4 & 1585028 & 5     & 1787955 & 202927 & 12.80\% & 1     & 1709978 & 124950 & 7.88\% & 5 \\
    b75\_50\_1 & 1252804 & 8     & 1374685 & 121881 & 9.73\% & 8     & 1343201 & 90397 & 7.22\% & 10 \\
    b75\_50\_2 & 1337446 & 9     & 1403629 & 66183 & 4.95\% & 5     & 1403629 & 66183 & 4.95\% & 5 \\
    b75\_50\_3 & 1249750 & 9     & 1368788 & 119038 & 9.52\% & 8     & 1368788 & 119038 & 9.52\% & 8 \\
    b75\_50\_4 & 1217508 & 9     & 1348203 & 130695 & 10.73\% & 10    & 1348203 & 130695 & 10.73\% & 10 \\
    c75\_50\_1 & 1310193 & 11    & 1390321 & 80128 & 6.12\% & 8     & 1375386 & 65193 & 4.98\% & 11 \\
    c75\_50\_2 & 1244255 & 10    & 1316595 & 72340 & 5.81\% & 11    & 1316595 & 72340 & 5.81\% & 11 \\
    c75\_50\_3 & 1201706 & 12    & 1386817 & 185111 & 15.40\% & 13    & 1386817 & 185111 & 15.40\% & 13 \\
    c75\_50\_4 & 1334782 & 11    & 1440836 & 106054 & 7.95\% & 5     & 1440836 & 106054 & 7.95\% & 5 \\
    \midrule
    a100\_75\_1 & 2286397 & 4     & 2476632 & 190235 & 8.32\% & 1     & 2459349 & 172952 & 7.56\% & 4 \\
    a100\_75\_2 & 2463187 & 3     & 2476632 & 13445 & 0.55\% & 1     & 2476632 & 13445 & 0.55\% & 1 \\
    a100\_75\_3 & 2415836 & 3     & 2476632 & 60796 & 2.52\% & 1     & 2467003 & 51167 & 2.12\% & 3 \\
    a100\_75\_4 & 2380150 & 4     & 2476632 & 96482 & 4.05\% & 1     & 2476632 & 96482 & 4.05\% & 1 \\
    b100\_75\_1 & 1950231 & 8     & 2061201 & 110970 & 5.69\% & 7     & 2061201 & 110970 & 5.69\% & 7 \\
    b100\_75\_2 & 2023097 & 8     & 2389395 & 366298 & 18.11\% & 1     & 2180286 & 157189 & 7.77\% & 8 \\
    b100\_75\_3 & 2062595 & 8     & 2133724 & 71129 & 3.45\% & 8     & 2133724 & 71129 & 3.45\% & 8 \\
    b100\_75\_4 & 1865323 & 9     & 1994265 & 128942 & 6.91\% & 7     & 1994265 & 128942 & 6.91\% & 7 \\
    c100\_75\_1 & 1843620 & 6     & 2107973 & 264353 & 14.34\% & 5     & 2090221 & 246601 & 13.38\% & 7 \\
    c100\_75\_2 & 1808867 & 11    & 2025331 & 216464 & 11.97\% & 9     & 2005071 & 196204 & 10.85\% & 13 \\
    c100\_75\_3 & 1820587 & 8     & 2019651 & 199064 & 10.93\% & 7     & 2019651 & 199064 & 10.93\% & 7 \\
    c100\_75\_4 & 1839007 & 9     & 2046525 & 207518 & 11.28\% & 8     & 2046525 & 207518 & 11.28\% & 8 \\
    \bottomrule
    \end{tabular}%
  \label{Heuristics}
\end{table}%

First, we show in Table \ref{Heuristics} the results for Algorithms Hs and Hc. As can be seen, we obtain better results with the second one. It finds the optimal solution for $6$ of the first $16$ problems and in the rest of the instances it is not too far from the optimal with an average gap of 4.65\%. All runs performed took less than 1 second whereas Xpress needs between 3 and 12 to solve these instances to optimality. The value obtained by Hc was used as an upper bound for the SG method, Algorithm~\ref{SGM}.

Table \ref{sgcr} shows the performance of the subgradient method SG (Algorithm \ref{SGM}) over the first $40$ data sets. The algorithm was stopped when the parameter $\beta$ was equal to $0$ with a starting value of $2$, and it was decreasing linearly at a rate of $0.005$ if a solution did not change after $30$ iterations. We also show the results for the first solution within a 10\% from LP(P).

After the time reported, SG did not obtain the value $LP(P)$ for any of the problems. However, as mentioned before, the SG will be useful to find initial multipliers for the dual ascend DA.

\begin{table}[h! ]
  \centering
      \caption{Computational results of subgradient method.}
  \footnotesize
\setlength{\tabcolsep}{3.4pt} 
\renewcommand{\arraystretch}{1.3} 
    \begin{tabular}{rrrrrrrrrrr}
    \toprule
        \multicolumn{1}{c}{} & \multicolumn{1}{c}{} & \multicolumn{5}{c}{\textbf{Until $\beta=0$}}  & \multicolumn{4}{c}{\textbf{First solution $<10\%$}} \\
        \cmidrule(l){3-7}
        \cmidrule(l){8-11}
    \multicolumn{1}{c}{\textbf{Prob}} & \multicolumn{1}{c}{\textbf{LP(P)}} & \multicolumn{1}{c}{\textbf{SG}} & \multicolumn{1}{c}{\textbf{itbsol}} & \multicolumn{1}{c}{\textbf{iter}} & \multicolumn{1}{c}{\textbf{t}} & \multicolumn{1}{c}{\textbf{GAP$_\text{LP}$}} & \multicolumn{1}{c}{\textbf{sol$10\%$}} & \multicolumn{1}{c}{\textbf{iter}} & \multicolumn{1}{c}{\textbf{t}} & \multicolumn{1}{c}{\textbf{GAP$_\text{LP}$}} \\
    \midrule
    131\_1 & 925492 & 881876 & 1474  & 1500  & 19    & 4.71\% & 833518 & 613   & 8     & 9.90\% \\
    131\_2 & 925195 & 872136 & 1450  & 1500  & 19    & 5.73\% & 833747 & 695   & 9     & 9.90\% \\
    131\_3 & 955447 & 929854 & 1488  & 1500  & 20    & 2.68\% & 860178 & 636   & 8     & 10.00\% \\
    131\_4 & 933025 & 903486 & 1462  & 1500  & 19    & 3.17\% & 841581 & 580   & 7     & 9.80\% \\
    132\_1 & 1007417 & 981624 & 1500  & 1500  & 19    & 2.56\% & 907222 & 479   & 6     & 9.90\% \\
    132\_2 & 990513 & 957530 & 1389  & 1425  & 18    & 3.33\% & 897874 & 445   & 6     & 9.40\% \\
    132\_3 & 1009054 & 974902 & 1500  & 1500  & 19    & 3.38\% & 911244 & 658   & 8     & 9.70\% \\
    132\_4 & 966305 & 910344 & 1388  & 1434  & 18    & 5.79\% & 869850 & 604   & 8     & 10.00\% \\
    133\_1 & 998199 & 958109 & 1395  & 1498  & 19    & 4.02\% & 898706 & 593   & 8     & 10.00\% \\
    133\_2 & 971719 & 947626 & 1443  & 1500  & 19    & 2.48\% & 882263 & 423   & 5     & 9.50\% \\
    133\_3 & 1023593 & 982907 & 1473  & 1500  & 19    & 3.97\% & 927636 & 565   & 7     & 9.40\% \\
    133\_4 & 1001253 & 948311 & 1489  & 1500  & 19    & 5.29\% & 902108 & 634   & 8     & 9.90\% \\
    134\_1 & 1226933 & 1013753 & 1192  & 1226  & 16    & 2.21\% & 933610 & 375   & 5     & 9.90\% \\
    134\_2 & 1041770 & 998540 & 1165  & 1239  & 16    & 4.15\% & 942697 & 530   & 7     & 9.50\% \\
    134\_3 & 1023070 & 994435 & 1155  & 1286  & 16    & 2.80\% & 921339 & 415   & 5     & 9.90\% \\
    134\_4 & 1050134 & 1003356 & 980   & 1026  & 13    & 4.45\% & 945217 & 439   & 6     & 10.00\% \\
\midrule
    a75\_50\_1 & 1201542 & 1169011 & 732   & 850   & 17    & 2.71\% & 1096795 & 106   & 2     & 9.40\% \\
    a75\_50\_2 & 1188359 & 1158500 & 829   & 896   & 18    & 2.51\% & 1076464 & 107   & 2     & 9.40\% \\
    a75\_50\_3 & 1189183 & 1167203 & 936   & 999   & 20    & 1.85\% & 1073799 & 157   & 3     & 9.70\% \\
    a75\_50\_4 & 1200068 & 1175405 & 854   & 937   & 19    & 2.06\% & 1094899 & 106   & 2     & 8.80\% \\
    b75\_50\_1 & 900617 & 879687 & 1118  & 1182  & 24    & 2.32\% & 812083 & 334   & 7     & 9.80\% \\
    b75\_50\_2 & 922048 & 896384 & 1111  & 1211  & 25    & 2.78\% & 830962 & 381   & 8     & 9.90\% \\
    b75\_50\_3 & 897073 & 882132 & 1208  & 1271  & 25    & 1.67\% & 808366 & 326   & 7     & 9.90\% \\
    b75\_50\_4 & 885392 & 856157 & 784   & 925   & 19    & 3.30\% & 801482 & 281   & 6     & 9.50\% \\
    c75\_50\_1 & 892837 & 879074 & 845   & 876   & 18    & 1.54\% & 804437 & 304   & 6     & 9.90\% \\
    c75\_50\_2 & 882933 & 870772 & 1120  & 1175  & 23    & 1.38\% & 797745 & 317   & 6     & 9.60\% \\
    c75\_50\_3 & 859591 & 852325 & 1013  & 1050  & 22    & 0.85\% & 786773 & 347   & 7     & 8.50\% \\
    c75\_50\_4 & 928064 & 891598 & 924   & 1084  & 22    & 3.93\% & 837696 & 415   & 8     & 9.70\% \\
\midrule
    a100\_75\_1 & 1800032 & 1725539 & 635   & 721   & 41    & 4.14\% & 1626622 & 99    & 6     & 9.60\% \\
    a100\_75\_2 & 1793377 & 1723484 & 679   & 728   & 42    & 3.90\% & 1618944 & 99    & 6     & 9.70\% \\
    a100\_75\_3 & 1788395 & 1724444 & 664   & 745   & 42    & 3.58\% & 1632328 & 103   & 6     & 8.70\% \\
    a100\_75\_4 & 1795298 & 1702235 & 482   & 678   & 38    & 5.18\% & 1620377 & 100   & 6     & 9.70\% \\
    b100\_75\_1 & 1330138 & 1289189 & 1010  & 1116  & 63    & 3.08\% & 1198182 & 187   & 11    & 9.90\% \\
    b100\_75\_2 & 1353824 & 1284338 & 652   & 691   & 39    & 5.13\% & 1225149 & 373   & 21    & 9.50\% \\
    b100\_75\_3 & 1351603 & 1306286 & 987   & 1044  & 59    & 3.35\% & 1217952 & 270   & 15    & 9.90\% \\
    b100\_75\_4 & 1332525 & 1258497 & 432   & 679   & 38    & 5.56\% & 1201093 & 157   & 6     & 9.90\% \\
    c100\_75\_1 & 1250876 & 1168284 & 510   & 710   & 40    & 6.60\% & 1127235 & 303   & 17    & 9.90\% \\
    c100\_75\_2 & 1239182 & 1154640 & 533   & 738   & 42    & 6.82\% & 1119427 & 424   & 24    & 9.70\% \\
    c100\_75\_3 & 1232462 & 1173522 & 670   & 804   & 46    & 4.78\% & 1115951 & 270   & 15    & 9.50\% \\
    c100\_75\_4 & 1243861 & 1194966 & 895   & 942   & 54    & 3.93\% & 1121515 & 356   & 20    & 9.80\% \\
    \bottomrule
    \end{tabular}%
    \label{sgcr}
\end{table}%

Finally, we can see the results for the algorithm ADA in Table \ref{ADAtable}. The routine needs parameters $sg\_iter$, $da\_iter$ and $vfh\_iter$, and different values were used for the four groups of data sets tested. These were empirically set as:\\

\noindent a75\_50, b75\_50, c75\_50: $sg\_iter=50$, $da\_iter=3$ and $fhv\_iter=2$,\\
a100\_75, b100\_75, c100\_75: $sg\_iter=100$, $da\_iter=7$ and $fhv\_iter=2$,\\
a125\_100, b125\_100, c125\_100: $sg\_iter=170$, $da\_iter=10$ and $fhv\_iter=2$,\\
a150\_100, b150\_100, c150\_100: $sg\_iter=170$, $da\_iter=12$ and $fhv\_iter=2$,

\noindent and for all the instances $ps=0.25$ (see Algorithm \ref{FVH}).\\

\begin{landscape}
\footnotesize
\setlength{\tabcolsep}{3.4pt} 
\renewcommand{\arraystretch}{1.15} 
\begin{longtable}{rrrrrrrrrrrrrrrrr}
\caption{Computational results for ADA.}\\
    \toprule
          & \multicolumn{4}{c}{\textbf{Xpress}} & \multicolumn{3}{c}{\textbf{Hc}} & \multicolumn{2}{c}{\textbf{SG}} & \multicolumn{4}{c}{\textbf{DA with VFH }} & \multicolumn{3}{c}{\textbf{ADA}} \\
        \cmidrule(l){2-5}
        \cmidrule(l){6-8}
        \cmidrule(l){9-10}
        \cmidrule(l){11-14}
        \cmidrule(l){15-17}
\multicolumn{1}{c}{\textbf{Prob}} & \multicolumn{1}{c}{\textbf{Optimal}} & \multicolumn{1}{c}{\textbf{LP}} & \multicolumn{1}{c}{\textbf{y\_j}} & \multicolumn{1}{c}{\textbf{t}} & \multicolumn{1}{c}{\textbf{bestUB}} & \multicolumn{1}{c}{\textbf{y\_j}} & \multicolumn{1}{c}{\textbf{t}} & \multicolumn{1}{c}{\textbf{LB}} & \multicolumn{1}{c}{\textbf{t}} & \multicolumn{1}{c}{\textbf{bestUB}} & \multicolumn{1}{c}{\textbf{LB}} & \multicolumn{1}{c}{\textbf{y\_j}} & \multicolumn{1}{c}{\textbf{t}} & \multicolumn{1}{c}{\textbf{Tt}} & \multicolumn{1}{c}{\textbf{imp t}} & \multicolumn{1}{c}{\textbf{GAP$_\text{o}\%$}} \\
\midrule
\endhead
\\
\multicolumn{17}{r}{{Continued on next page}}
\endfoot
\endlastfoot
    a75\_50\_1 & 1661269 & 1201542 & 7     & 16    & 1787955 & 1     & 0     & 832797 & 4     & 1662877 & 1290126 & 5     & 21    & 25    & -51\% & 0.10\% \\
    a75\_50\_2 & 1632907 & 1188359 & 6     & 21    & 1779576 & 4     & 0     & 828462 & 4     & 1648421 & 1261845 & 6     & 28    & 31    & -47\% & 0.95\% \\
    a75\_50\_3 & 1632213 & 1189183 & 7     & 18    & 1738404 & 3     & 0     & 783044 & 3     & 1632866 & 1286922 & 7     & 24    & 27    & -51\% & 0.04\% \\
    a75\_50\_4 & 1585028 & 1200068 & 5     & 18    & 1709978 & 5     & 0     & 863178 & 3     & 1585028 & 1250104 & 5     & 28    & 31    & -72\% & 0.00\% \\
    b75\_50\_1 & 1252804 & 900617 & 8     & 11    & 1343201 & 10    & 0     & 355739 & 2     & 1252804 & 951766 & 8     & 28    & 30    & -165\% & 0.00\% \\
    b75\_50\_2 & 1337446 & 922048 & 9     & 20    & 1403629 & 5     & 0     & 376669 & 3     & 1337446 & 954023 & 9     & 35    & 38    & -94\% & 0.00\% \\
    b75\_50\_3 & 1249750 & 897073 & 9     & 17    & 1368788 & 8     & 0     & 348636 & 3     & 1255947 & 975724 & 11    & 39    & 42    & -141\% & 0.50\% \\
    b75\_50\_4 & 1217508 & 885392 & 9     & 12    & 1348203 & 10    & 0     & 414596 & 3     & 1222561 & 978717 & 10    & 35    & 38    & -232\% & 0.42\% \\
    c75\_50\_1 & 1310193 & 892837 & 11    & 17    & 1375386 & 11    & 0     & 550106 & 4     & 1310193 & 972827 & 11    & 38    & 42    & -151\% & 0.00\% \\
    c75\_50\_2 & 1244255 & 882933 & 10    & 11    & 1316595 & 11    & 0     & 533586 & 3     & 1244255 & 964720 & 10    & 36    & 39    & -244\% & 0.00\% \\
    c75\_50\_3 & 1201706 & 859591 & 12    & 8     & 1386817 & 13    & 0     & 549787 & 4     & 1201706 & 934664 & 12    & 33    & 37    & -341\% & 0.00\% \\
    c75\_50\_4 & 1334782 & 928064 & 11    & 21    & 1440836 & 5     & 0     & 581078 & 4     & 1356714 & 974828 & 12    & 41    & 45    & -111\% & 1.64\% \\
\midrule
    a100\_75\_1 & 2286397 & 1800032 & 4     & 61    & 2459349 & 4     & 1     & 1649771 & 16    & 2321812 & 1950826 & 4     & 74    & 91    & -49\% & 1.55\% \\
    a100\_75\_2 & 2463187 & 1793377 & 3     & 131   & 2476632 & 4     & 1     & 1630225 & 16    & 2476632 & 1982594 & 3     & 106   & 122   & 7\%   & 0.55\% \\
    a100\_75\_3 & 2415836 & 1788395 & 3     & 148   & 2467003 & 3     & 1     & 1634629 & 16    & 2415836 & 1970116 & 3     & 99    & 116   & 22\%  & 0.00\% \\
    a100\_75\_4 & 2380150 & 1795298 & 4     & 124   & 2476632 & 3     & 1     & 1628551 & 17    & 2386981 & 1938879 & 4     & 99    & 116   & 7\%   & 0.29\% \\
    b100\_75\_1 & 1950231 & 1330138 & 8     & 624   & 2061201 & 7     & 1     & 1092987 & 17    & 1984720 & 1511894 & 10    & 125   & 142   & 77\%  & 1.77\% \\
    b100\_75\_2 & 2023097 & 1353824 & 8     & 727   & 2180286 & 8     & 1     & 1150334 & 18    & 2058495 & 1517533 & 11    & 143   & 161   & 78\%  & 1.75\% \\
    b100\_75\_3 & 2062595 & 1351603 & 8     & 841   & 2133724 & 8     & 1     & 1165120 & 18    & 2062595 & 1536082 & 8     & 268   & 285   & 66\%  & 0.00\% \\
    b100\_75\_4 & 1865323 & 1332525 & 9     & 546   & 1994265 & 7     & 1     & 1143446 & 15    & 1886598 & 1512185 & 7     & 141   & 155   & 72\%  & 1.14\% \\
    c100\_75\_1 & 1843620 & 1250876 & 6     & 430   & 2090221 & 7     & 1     & 1077180 & 20    & 1843620 & 1456543 & 6     & 215   & 235   & 45\%  & 0.00\% \\
    c100\_75\_2 & 1808867 & 1239182 & 11    & 396   & 2005071 & 13    & 1     & 1053868 & 19    & 1815373 & 1415172 & 9     & 171   & 190   & 52\%  & 0.36\% \\
    c100\_75\_3 & 1820587 & 1232462 & 8     & 423   & 2019651 & 7     & 1     & 1054526 & 16    & 1820587 & 1405713 & 8     & 194   & 210   & 50\%  & 0.00\% \\
    c100\_75\_4 & 1839007 & 1243861 & 9     & 583   & 2046525 & 8     & 1     & 1064170 & 21    & 1839007 & 1417402 & 9     & 227   & 248   & 57\%  & 0.00\% \\
\midrule
    a125\_100\_1 & 3041451 & 2392412 & 2     & 257   & 3070535 & 1     & 1     & 2235753 & 47    & 3070535 & 2539124 & 2     & 151   & 198   & 23\%  & 0.96\% \\
    a125\_100\_2 & 3040248 & 2393448 & 2     & 302   & 3070535 & 1     & 1     & 2232859 & 48    & 3070535 & 2629291 & 2     & 201   & 249   & 18\%  & 1.00\% \\
    a125\_100\_3 & 3055260 & 2362216 & 3     & 325   & 3070535 & 1     & 1     & 2227182 & 47    & 3070535 & 2593865 & 3     & 209   & 256   & 21\%  & 0.50\% \\
    a125\_100\_4 & 3056428 & 2381167 & 2     & 334   & 3070535 & 1     & 1     & 2244017 & 54    & 3070535 & 2529404 & 2     & 208   & 262   & 22\%  & 0.46\% \\
    b125\_100\_1 & 2640798 & 1794710 & 7     & 5297  & 2850664 & 5     & 1     & 1601985 & 58    & 2640798 & 2082726 & 7     & 989   & 1047  & 80\%  & 0.00\% \\
    b125\_100\_2 & 2550592 & 1790869 & 7     & 2086  & 2808259 & 9     & 1     & 1600582 & 56    & 2568522 & 2046667 & 7     & 848   & 903   & 57\%  & 0.70\% \\
    b125\_100\_3 & 2604906 & 1792133 & 4     & 4059  & 2782609 & 5     & 1     & 1607076 & 58    & 2604906 & 1977647 & 4     & 435   & 494   & 88\%  & 0.00\% \\
    b125\_100\_4 & 2580595 & 1792581 & 7     & 2686  & 2778713 & 4     & 1     & 1587315 & 60    & 2637611 & 1975516 & 7     & 358   & 419   & 84\%  & 2.21\% \\
    c125\_100\_1 & 2491714 & 1663455 & 10    & 7805  & 2669965 & 9     & 1     & 1491505 & 61    & 2491714 & 1953129 & 10    & 1685  & 1746  & 78\%  & 0.00\% \\
    c125\_100\_2 & 2468480 & 1674102 & 9     & 6296  & 2674261 & 8     & 1     & 1487416 & 58    & 2518055 & 1957245 & 10    & 519   & 577   & 91\%  & 2.01\% \\
    c125\_100\_3 & 2559381 & 1672005 & 8     & 11381 & 2685807 & 7     & 1     & 1496552 & 59    & 2559381 & 1957575 & 8     & 1177  & 1236  & 89\%  & 0.00\% \\
    c125\_100\_4 & 2538550 & 1691148 & 8     & 6407  & 2683676 & 9     & 1     & 1503317 & 61    & 2561098 & 2016969 & 8     & 726   & 786   & 88\%  & 0.89\% \\
\midrule
    a150\_100\_1 & 3768087 & 2906284 & 1     & 371   & 3768087 & 1     & 2     & 2691219 & 57    & 3768087 & 3200250 & 3     & 338   & 395   & -6\%  & 0.00\% \\
    a150\_100\_2 & 3768087 & 2891681 & 1     & 457   & 3768087 & 1     & 1     & 2695660 & 56    & 3768087 & 3212545 & 2     & 329   & 384   & 16\%  & 0.00\% \\
    a150\_100\_3 & 3741364 & 2882917 & 3     & 340   & 3768087 & 1     & 1     & 2691438 & 58    & 3741364 & 3185719 & 3     & 306   & 363   & -7\%  & 0.00\% \\
    a150\_100\_4 & 3768087 & 2911017 & 1     & 591   & 3768087 & 1     & 1     & 2697003 & 63    & 3768087 & 3241759 & 2     & 337   & 400   & 32\%  & 0.00\% \\
    b150\_100\_1 & 3271859 & 2159537 & 4     & 15483 & 3487710 & 4     & 1     & 1916485 & 68    & 3271859 & 2651812 & 4     & 2357  & 2426  & 84\%  & 0.00\% \\
    b150\_100\_2 & 3227987 & 2160753 & 5     & 9677  & 3457623 & 5     & 2     & 1935777 & 68    & 3227987 & 2546429 & 5     & 975   & 1043  & 89\%  & 0.00\% \\
    b150\_100\_3 & 3150075 & 2148658 & 6     & 5986  & 3390435 & 8     & 1     & 1920876 & 69    & 3150075 & 2582861 & 6     & 1081  & 1150  & 81\%  & 0.00\% \\
    b150\_100\_4 & 3342783 & 2190488 & 5     & 9596  & 3637438 & 8     & 1     & 1927555 & 69    & 3342783 & 2602414 & 5     & 2225  & 2294  & 76\%  & 0.00\% \\
    c150\_100\_1 & 2979389 & 1988908 & 9     & 10008 & 3196861 & 5     & 1     & 1790894 & 71    & 2979389 & 2408302 & 9     & 1520  & 1591  & 84\%  & 0.00\% \\
    c150\_100\_2 & 3109105 & 1985389 & 7     & 21734 & 3291990 & 6     & 2     & 1766696 & 72    & 3109105 & 2404613 & 7     & 949   & 1021  & 95\%  & 0.00\% \\
    c150\_100\_3 & 2937767 & 1982388 & 8     & 10590 & 3079664 & 4     & 1     & 1770318 & 71    & 2954519 & 2379867 & 7     & 2369  & 2440  & 77\%  & 0.57\% \\
    c150\_100\_4 & 3165327 & 1997014 & 10    & 65134 & 3189247 & 7     & 1     & 1775186 & 75    & 3176587 & 2424997 & 12    & 5842  & 5917  & 91\%  & 0.36\% \\
    \bottomrule
  \label{ADAtable}%
\end{longtable}%
\end{landscape}

With the exception of the first group ($m=75$, $n=50$) and a few instances in the rest, the times obtained by ADA when it meets the optimal solution are considerably lower than those obtained by Xpress. It can be observed particularly in the last group, For example, the instance c\_150\_100\_2 had a reduction of $95\%$ on the total time. ADA is always close to the optimal values, in fact, the GAP average is only $0.43\%$.

\section{Conclusions} \label{conclusions}

In this paper we have proposed a heuristic method to solve the SPLPO inspired by techniques introduced in \citep{Cornuejols1977,Beltran2009}. The assignment and VUBs constraints in SPLPO were relaxed to formulate a Lagrangean problem. We solved its dual with a subgradient method and used a vector of multipliers as a starting point of an ascent algorithm for the dual of a semi-Lagrangean formulation.  Nevertheless, a better starting point should be the multipliers obtained by relaxing only the assignment constraints, the same family of constraints relaxed in our proposed SPLPO semi-Lagrangean problem. However, the sequence of linear subproblems that need to be solved in the subgradient method are not as easy to solve as those when our proposed relaxation is used. We used the variable fixing heuristic VFH because MIP problems in ADA are becoming increasingly difficult as the number of iterations grows. We have shown that the ADA algorithm works particularly well on large instances, but there must be a future discussion about the parameters settings.

\section{Acknowledgements}
This work has been supported in part by SENESCYT-Ecuador (National Secretary of Higher Education, Science, Technology and Innovation of Equator).
\noindent The research of Sergio Garc\'ia has been funded by Fundaci\'on S\'eneca (project 19320/PI/14).

\bibliography{bibliography}

\begin{thebibliography}{}

\bibitem[Beasley, 1990]{Beasley}
Beasley, E. (1990).
\newblock {OR}-library: Distributing test problems by electronic mail.
\newblock {\em Journal of the Operational Research Society}, 41(11):1069--1072.

\bibitem[Beltr\'an et~al., 2006]{Beltran2006}
Beltr\'an, C., Tandoki, C., and Vial, J. (2006).
\newblock Solving the p-median problem with a semi-{Lagrangian} relaxation.
\newblock {\em Computational Optimization and Applications}, 35:239--260.

\bibitem[Beltr\'an et~al., 2012]{Beltran2009}
Beltr\'an, C., Vial, J., and Alonso, A. (2012).
\newblock Semi-{Lagrangian} relaxation applied to the uncapacited facility
  location problem.
\newblock {\em Computational Optimization and Applications}, 51:387--409.

\bibitem[Bilde and Krarup, 1977]{Bilde1977}
Bilde, O. and Krarup, J. (1977).
\newblock Sharp lower bounds and efficient algorithms for the simple plant
  location problem.
\newblock {\em Annals of Discrete Mathematics}, 1:79--97.

\bibitem[C\'anovas et~al., 2006]{Canovas}
C\'anovas, L., Garc\'ia, S., Labb\'e, M., and Mar\'in, A. (2006).
\newblock A strengthened formulation for the simple plant location problem with
  order.
\newblock {\em Operations Research Letters}, 35:141--150.

\bibitem[Conforti et~al., 2014]{IntegerProgramming}
Conforti, M., Cornu\'ejols, G., and Zambelli, G. (2014).
\newblock {\em Integer programming}.
\newblock Springer.

\bibitem[Cornu\'ejols et~al., 1977]{Cornuejols1977}
Cornu\'ejols, G., Fisher, M., and Nemhauser, G. (1977).
\newblock Location of bank accounts to optimize float: an analytic study of
  axact and approximated algorithms.
\newblock {\em Management Science}, 23(8):789--810.

\bibitem[Fisher, 2004]{Fisher2004}
Fisher, M.~L. (2004).
\newblock The {Lagrangian} relaxation method for solving integer programming
  problems algorithms.
\newblock {\em Management Science}, 50(12):1861--1871.

\bibitem[Geoffrion, 1974]{Geoffrion1974}
Geoffrion, A. (1974).
\newblock Lagrangian relaxation for integer programming.
\newblock {\em Mathematical Programming Study}, 2:82--114.

\bibitem[Guignard, 2003]{GuignardTutorial}
Guignard, M. (2003).
\newblock {Lagrangean} relaxation. {A} tutorial.
\newblock {\em TOP}, 11(2):151--228.

\bibitem[Held and Karp, 1971]{Held1971}
Held, M. and Karp, R. (1971).
\newblock The traveling salesman problem and minimum spanning trees: part {II}.
\newblock {\em Mathematical Programming}, 1:6--25.

\bibitem[Held et~al., 1974]{Held1974}
Held, M., Wolfe, P., and Crowder, H. (1974).
\newblock Validation of subgradient optimization.
\newblock {\em Mathematical Programming}, 6:62--88.

\bibitem[J\"{o}rnsten, 2016]{Jornsten2016}
J\"{o}rnsten, K. (2016).
\newblock An improved {Lagrangian} relaxation and dual ascent approach to
  facility location problems.
\newblock {\em Computational Management Science}, 13:317--348.

\bibitem[Monabbati, 2014]{Monabbati}
Monabbati, E. (2014).
\newblock An application of a {Lagrangian-type} relaxation for the
  uncapacitated facility location problem.
\newblock {\em Japan Journal of Industrial and Applied Mathematics},
  31:483--499.

\bibitem[Poljak, 1967]{Poljak}
Poljak, B.~T. (1967).
\newblock A general method for solving extremum problems.
\newblock {\em Soviet Mathematics Doklady}, 8:593--597.

\bibitem[Vasilyev and Klimentova, 2010]{Vasilyev2010}
Vasilyev, I. and Klimentova, X. (2010).
\newblock The branch and cut method for the facility location problem with
  client's preferences.
\newblock {\em Journal of Applied and Industrial Mathematics}, 4(3):441--454.

\bibitem[Vasilyev et~al., 2013]{Vasilyev2013}
Vasilyev, I.~L., Klimentova, X., and Boccia, M. (2013).
\newblock Polyhedral study of simple plant location problem with order.
\newblock {\em Operations Research Letters}, 41:153--158.

\end{thebibliography}
\bibliographystyle{apalike}

\end{document}